\documentclass[pdflatex,sn-vancouver-ay]{sn-jnl}%

\usepackage{graphicx}
\usepackage{algorithmic}
\usepackage{algorithm}

\usepackage{amsmath}
\usepackage{amssymb}
\usepackage{amsthm}
\usepackage{amsfonts}
\usepackage{amsxtra}
\usepackage{enumitem}
\usepackage{newtxtext}
\usepackage{newtxmath}

\theoremstyle{plain}
\newtheorem{theorem}{Theorem}

\newtheorem{fact}[theorem]{Fact}
\newtheorem{lemma}[theorem]{Lemma}
\newtheorem{corollary}[theorem]{Corollary}
\theoremstyle{definition}
\newtheorem{definition}[theorem]{Definition}
\newtheorem{example}[theorem]{Example}
\theoremstyle{remark}
\newtheorem{remark}[theorem]{Remark}
\numberwithin{theorem}{section}
\numberwithin{equation}{section}

\usepackage[svgnames]{xcolor}
\definecolor{darkgreen}{rgb}{0,.35,0}
\definecolor{darkblue}{rgb}{0,0,.1}
\definecolor{darkred}{rgb}{.6,0,0}

\usepackage{hyperref}
\hypersetup{pdfauthor={Corless, Giesbrecht, Rafiee Sevyeri, and Saunders},
  pdftitle={On Parametric Linear System Solving},
  bookmarksnumbered=true,colorlinks,linkcolor=darkblue,citecolor=darkgreen,
  urlcolor=darkred}

\usepackage[textsize=tiny,color=yellow]{todonotes}

\newcommand{\mat}[1]{\ensuremath{\mathbf{#1}}}
\renewcommand{\vec}[1]{\ensuremath{\mathbf{#1}}}

\newcommand{\den}{\mathrm{den}}
\newcommand{\num}{\mathrm{num}}

\newcommand{\rank}{\mathrm{rank}}
\newcommand{\lcm}{\mathrm{lcm}}

\newcommand{\sqfr}{\mathrm{sqfr}}
\newcommand{\bigO}{\mathrm{O}}
\newcommand{\softO}{{O\,\tilde{}\,}}
\newcommand{\RR}{\mathbb{R}}
\newcommand{\F}{{\mathsf{F}}}
\newcommand{\K}{{\mathsf{K}}}

\begin{document}
\title{On Parametric Linear System Solving}

\author[12]{\fnm{Robert M.} \sur{Corless}}\email{rcorless@uwo.ca}

\author*[2]{\fnm{Mark} \sur{Giesbrecht}}\email{mwg@uwaterloo.ca}

\author[2]{\fnm{Leili} \sur{Rafiee Sevyeri}}\email{leili.rafiee.sevyeri@uwaterloo.ca}

\author[3]{\fnm{B. David} \sur{Saunders}}\email{saunders@udel.edu}

\affil[1]{\orgdiv{Ontario Research Centre for Computer Algebra, School of Mathematical and Statistical Sciences},
\orgname{University of Western Ontario},
\orgaddress{\city{London}, \state{ON}, \country{Canada}}}

\affil[2]{\orgdiv{Cheriton School of Computer Science},
\orgname{University of Waterloo},
\orgaddress{\city{Waterloo}, \state{ON}, \postcode{N2L 3G1}, \country{Canada}}}

\affil[3]{\orgdiv{Department of Computer and Information Sciences},
\orgname{University of Delaware},
\orgaddress{\city{Newark}, \state{Delaware}, \country{USA}}}

\abstract{
  Parametric linear systems are linear systems of equations in which
  some symbolic parameters, that is, symbols that are not considered
  to be candidates for elimination or solution in the course of
  analyzing the problem, appear in the coefficients of the system.  In
  this paper we assume that the symbolic parameters appear
  polynomially in the coefficients and that the only variables to be
  solved for are those of the linear system.  The consistency of the
  system and expression of the solutions may vary depending on the
  values of the parameters.  It is well-known that it is possible to
  specify a covering set of regimes, each of which is a
  Zariski-constructible condition on the parameters together with a
  solution description valid under that condition.

  We provide a method of solution that requires time polynomial in the
  matrix dimension and the degrees of the polynomials when there are
  up to three parameters.  We also discuss examples suggesting how the
  method may be useful beyond the formal three-parameter setting.
  
  In previous methods the number of regimes needed is exponential in
  the system dimension and polynomial degree of the parameters.  Our
  approach exploits the Hermite and Smith normal forms that may be
  computed when the system coefficient domain is mapped to the
  univariate polynomial domain over suitably constructed fields.  Our
  method identifies \emph{intrinsic singularities} and
  \emph{ramification points} where the algebraic and geometric
  structure of the matrix changes.

  Parametric eigenvalue problems are addressed as well: we treat
  $\lambda$ as a parameter in addition to those in $\mat{A}$ and solve
  the parametric system $(\lambda \mat{I} - \mat{A})\vec{u} = 0$. The
  algebraic conditions on $\lambda$ required for a nontrivial
  nullspace define the eigenvalues. We do not directly address the
  problem of computing the Jordan form, but our approach allows the
  construction of the algebraic and geometric eigenvalue
  multiplicities revealed by the Frobenius form, which is a key step
  in the construction of the Jordan form of a matrix.
}

\maketitle 

\section{Introduction}
\label{sec:intro}

In broad generality, symbolic computation is concerned with
mathematical equations that contain \emph{symbols}; symbols are used
both for \emph{variables}, which are typically to be solved for, and
\emph{parameters}, which are typically carried through and appear in
the \emph{solutions}, which are then interpreted as formulae: that is,
objects that can be further studied, perhaps by varying the
parameters.  One prominent early researcher said that the difference
between symbolic and numeric computation was merely a matter of
\emph{when} numerical values were inserted into the parameters: before
the computation meant you were going to do things numerically, and
after the computation meant you had done symbolic computation.  The
words ``parameters'' and ``variables'' are therefore not precisely
descriptive, and can often be used interchangeably.  Indeed as a
matter of practice, polynomial equations can often be taken to have
one subset of its symbols taken as variables rather than any other
subset in quite strategic fashion: it may be better to solve for $x$
as a function of $y$ than to solve for $y$ as a function of $x$.

In this paper we are concerned with systems of equations containing
several symbols, some of which we take to be variables, and all the
rest as parameters.  In addition, we restrict our attention to
problems in which the \emph{variables} appear only linearly.
Parameters are allowed to appear polynomially, of whatever degree.
  
Parametric Linear Systems (PLS) arise in many contexts, for instance
in the analysis of the stability of equilibria in dynamical systems
models such as occur in mathematical biology and other areas.
Understanding the different potential kinds of dynamical behavior can
be important for model selection as well as analysis.  Another
important area of interest is the role of parametric linear systems in
dealing with the stability of the equilibria of parametric autonomous
systems of ordinary differential equations (see~\citep{Sit92}
and~\citep{Gol87}).  One particularly famous example is the
Lotka-Volterra system which arises naturally from predator-prey
equations.  See also~\citep{SavVoi87a} and~\citep{SavVoi87b}.  Other
examples of the use of parametric linear system from science and
engineering include their application in computing the characteristic
solutions for differential equations~\citep{DauLio88}, dealing with
colored Petri nets~\citep{Jen97} and in operations research and
engineering~\citep{DesTho01}, \citep{Kol02}, \citep{MuhMul01},
\citep{WinMei84}.  Some problems in robotics \citep{Buc88} and certain
modelling problems in mathematical biology, see e.g.,
\citep{WahBet19}, also can benefit from the ability to effectively
solve parametric linear systems.

After some discussion of prior comprehensive solving work in
Section~\ref{sec:prevwork}, we proceed with formal problem and
solution definitions for parametric linear systems in Section
\ref{sec:defs}. Our primary tool for solving these is by way of
\emph{Comprehensive Triangular Smith Normal Form (CTSNF)}, which is
introduced in Section \ref{sec:pls2snf}, where we also reduce PLS to
CTSNF. Section \ref{sec:solve} describes the solution of CTSNF
problems for the case of up to three parameters.

An application that seems at first to be of only theoretical interest
is the computation of the \emph{matrix logarithm}, or indeed any of
several other matrix functions such as matrix square root.  We briefly
discuss this example in more detail with a pair of small matrices in
Section~\ref{sec:matrixlogarithm}.  We also give other examples in
Section~\ref{sec:nfeigs}.

A preliminary conference version of this work appeared in
\cite{CorGie20}.  The present version expands the definitions, proofs,
and treatment of constructible regimes.

\section{Previous Work on Parameterized Linear Systems}
\label{sec:prevwork}

Interest in computation of the solution of parameterized linear
systems dates back to the beginning of symbolic computation.  For
instance, one of the first things users have requested of computer
algebra systems is the explicit form of the inverse of a matrix
containing only symbolic entries\footnote{This is merely an anecdote,
  but one of the present authors attests that this really has
  happened.}: the user is then typically quite dissatisfied at the
complexity of the answer if the dimension is greater than, say,
three. Of course, the determinant itself, which must appear in such an
answer, has a factorial number of terms in it, and thus growth in the
size of the answer must be more than exponential.  Therefore the
complexity of any algorithm to solve parameterized linear systems must
be at least exponential in the number of parameters.

An interesting pair of papers addressing the case of only one
parameter is~\citep{BoyKal14} and~\citep{KalPer17}.  These papers
assume full rank of the linear system, and thus compute the
``generic'' case when in fact there are isolated values of the
parameter for which the rank drops, and use rational interpolation of
the numerical solutions of specialized linear systems to recover this
generic solution.

Many authors have sought comprehensive solutions, by which is meant
complete coverage of all parametric regimes, through various means.
One of the first explicit methods was the matrix-minor based approach
of \cite{Sit92}, which enables practical solution of many problems of
interest.  Recently, the problem of computing the Jordan form of a
parametric matrix once the Frobenius form is known has been approached
using Regular Chains~\citep{CorMor17}, and this has been moderately
successful in practice.  Simple methods and heuristics for linear
systems containing parameters continue to generate interest, even when
Regular Chains are used, such as in~\citep{CamMor20}.

Other authors such as~\cite{Wei92}, \cite{LemMaz05}, \cite{MonWib10},
\cite{KapSun13}, and~\cite{Mon18} have tackled the even more difficult
problem of computing the comprehensive solution of systems of
\emph{polynomial} equations containing parameters, and of course their
methods can be applied to the linear equations being considered here.

By restricting our attention in this paper to linear problems and to
those of a constant number of parameters (e.g., three or fewer) we are
able to guarantee better worst case performance (polynomially many
solution regimes) and hope to provide better efficiency in many
instances than is possible using those general-purpose approaches.

\section{Definitions and Notation}
\label{sec:defs}

Let $\F$ be a field and $Y = (y_1, \ldots, y_s)$ a list of parameters,
and fix an algebraic closure $\overline{\F}$ of $\F$.  Then $\F[Y]$ is
the ring of polynomials and $\F(Y)$ is the field of rational functions
in $Y$.  For each tuple $a = (a_1, \ldots, a_s)$ in $\overline{\F}^s$,
evaluation at $a$ is a mapping $\F[Y] \rightarrow \overline{\F}$.  We
will extend this mapping componentwise to polynomials, vectors,
matrices, and sets thereof over $\F[Y]$ (i.e., in $\F[Y][x]=\F[Y,x]$).
Equivalently, one may interpret evaluation in arbitrary field
extensions of $\F$; we fix $\overline{\F}$ for notational convenience.

In many of our constructions an additional symbol $x$ will play a
special role: it is a parameter of the PLS problems but, for purposes
of computing Hermite/Smith forms, we treat $x$ as the univariate
polynomial variable over $\F(Y)$.  Accordingly, evaluation at $Y=a$
extends to a \emph{partial} evaluation map
$\F[Y,x]\rightarrow \overline{\F}[x]$ by substituting $Y=a$ and
leaving $x$ unbound.  For $f\in \F[Y,x]$ we write $f(a,x)$ for the
resulting polynomial in $x$, and for $\xi\in\overline{\F}$ we write
$f(a,\xi)$ (equivalently $f(a,x)\vert_{x=\xi}$) for its value under
the further specialization $x=\xi$.

We use the Householder convention, typesetting matrices in upper case
bold, e.g.~\mat{A}, and lower case bold for vectors, e.g.~\vec{b}.

For the most part, for such objects over $\F[Y,x]$, we know $Y$ and
$x$ from context and write $\mat{A}$ rather than $\mat{A}(Y,x)$, but
write $\mat{A}(a)\in\overline{\F}[x]^{m\times n}$ for the partial
evaluation at $Y=a$, and $\mat{A}(a,\xi)\in\overline{\F}^{m\times n}$
for the full evaluation at $(Y,x)=(a,\xi)$.

For a set of polynomials $S\subseteq \F[Y]$, we will denote by $V(S)$
the variety of the ideal generated by $S$ in $\overline{\F}^s$.  This
is the set of tuples $a\in\overline{\F}^s$ such that $f(a)=0$, for all
$f \in S$.  For a polynomial $g \in \F[Y]$ let $D(g)$ denote its
principal Zariski-open set,
$D(g)=\{a\in \overline{\F}^s\mid g(a)\neq 0\} =
\overline{\F}^s\setminus V(g)$.  For a finite set $N\subseteq \F[Y]$
define $D(N)=\bigcap_{n\in N}D(n)=D\big(\prod_{n\in N}n\big)$.

When polynomial sets also involve $x$, we use the same notation with
respect to the extended parameter space
$\overline{\F}^s\times \overline{\F}$.  For example, for
$S\subseteq \F[Y,x]$ we write $V(S)$ for the set of pairs $(a,\xi)$
such that $f(a,\xi)=0$ for all $f\in S$, and for $g\in \F[Y,x]$ we
write $D(g)=\{(a,\xi)\mid g(a,\xi)\neq 0\}$ and
$D(N)=D\big(\prod_{n\in N}n\big)$ for finite $N\subseteq \F[Y,x]$.  We
will be concerned with basic Zariski-constructible (locally closed)
sets of the form $V(Z)\cap D(N)$.  We use Zariski-constructible rather
than ``semi-algebraic'' since we do not assume that $\F$ is an ordered
field.

Our inputs are polynomial in the parameters but the output
coefficients in general are rational functions.  The evaluation
mapping extends partially to $\F(Y,x)$: for a rational function
$r(Y,x)=n(Y,x)/d(Y,x)$ in lowest terms ($n$ and $d$ relatively prime),
define $\den(r)=d$; then $r(a,\xi)$ is well defined for
$(a,\xi)\in D(\den(r))$.  We extend $\den(\cdot)$ componentwise to
vectors and matrices over $\F(Y,x)$ (e.g., by taking the least common
multiple of entry denominators).  Throughout, whenever regime data
contains rational function coefficients, we explicitly adjoin the
relevant denominator factors to the nonvanishing constraint set $N$ so
that all expressions are well defined on the corresponding
constructible set.

When a condition is imposed on an element $q$ of a rational function
field such as $\F_Y(x)$, we use the following convention.  After
writing $q=\num(q)/\den(q)$ in lowest terms, the condition $q=0$
means $\num(q)=0$ together with $\den(q)\neq 0$, and the condition
$q\neq 0$ means $\num(q)\neq 0$ together with $\den(q)\neq 0$.
Thus writing $q$ in a vanishing set is shorthand for adjoining
$\num(q)$ to $Z$ and the factors of $\den(q)$ to $N$, while writing
$q$ in a nonvanishing set is shorthand for adjoining the factors of
both $\num(q)$ and $\den(q)$ to $N$.  The same convention is applied
componentwise to vectors and matrices.  For regimes over an algebraic
parameterized extension, these conditions are interpreted in the
corresponding quotient coordinate ring; equivalently, after choosing a
basis for the quotient, one may clear denominators and equate the
corresponding coordinate polynomials over the base parameter ring.

\begin{definition}
  \label{def:fac}
  For a polynomial $\Delta$ in a polynomial ring over $\F$ (e.g., in
  $\F[Y]$ or $\F[Y,x]$) we write $\mathrm{Fac}(\Delta)$ for the
  (finite) set of monic irreducible factors of $\Delta$ in that ring.
  For rational matrices/vectors with entries in $\F(Y,x)$ we write
  $\Delta(\mat{M}) := \den(\mat{M})$ and
  $\mathrm{Fac}(\mat{M}) := \mathrm{Fac}(\Delta(\mat{M}))$.
\end{definition}

\begin{definition}
  \label{def:welldefined}
  A PLS instance $(\mat{A},\vec{b},N,Z)$ is \textbf{well defined} if
  all entries of $\mat{A}$ and $\vec{b}$ are well defined on
  $V(Z)\cap D(N)$; equivalently,
  $V(Z)\cap D(N)\subseteq D(\den(\mat{A},\vec{b}))$.
\end{definition}

\begin{definition}
  The data for a PLS problem is a matrix $\mat{A}$ and right hand side
  vector $\vec{b}$ over $\F[Y,x]$, together with a
  Zariski-constructible constraint, $V(Z)\cap D(N)$, with
  $N,Z \subseteq \F[Y,x]$.  We are only interested in those parameter
  value tuples $(a,\xi)\in \overline{\F}^s\times\overline{\F}$ in
  $V(Z)\cap D(N)$, i.e., on which the polynomials in $N$ are nonzero
  and the polynomials in $Z$ are zero.
  %

  For the PLS problem $(\mat{A}, \vec{b}, N, Z)$:
  \begin{itemize}
  \item a \textbf{solution regime} is a tuple
    $(\vec{u}, \mat{B}, N', Z')$, with entries of $\vec{u}$ and
    $\mat{B}$ in $\F_Y(x)$ for some parameterized extension $\F_Y$ of
    $\F$ (Definition~\ref{def:paramext}), such that all denominator
    factors occurring in $\vec{u}$ and $\mat{B}$ are contained in
    $N'$, and such that, for all $(a,\xi) \in V(Z')\cap D(N')$,
    $\vec{u}(a,\xi)$ is a solution vector and $\mat{B}(a,\xi)$ is a
    matrix whose columns form a nullspace basis for $\mat{A}(a,\xi)$.

  \item An \textbf{inconsistency regime} is a triple $(\bot,N',Z')$
    such that, for all
    $(a,\xi) \in V(Z')\cap D(N')$, the specialized system
    $\mat{A}(a,\xi)\vec{u}=\vec{b}(a,\xi)$ has no solution. The $\bot$ symbol is to remind that no solution vector is on offer in this case.

  \item A \textbf{PLS solution} is a set of regimes (solution regimes
    and inconsistency regimes) that covers $V(Z)\cap D(N)$, which
    means every parameter value assignment that satisfies the problem
    Zariski-constructible constraint $N,Z$ also satisfies at least one
    regime Zariski-constructible constraint.  In other words
    \[
      V(Z)\cap D(N) ~\subseteq~ \bigcup_{i=1}^k V(Z_i)\cap D(N_i).
    \]
  \end{itemize}
\end{definition}

We call entries that \emph{must} occur in any $Z$ in the solution an
\emph{intrinsic restriction}, or singularity. We call the differing
sets $V(Z_i)\cap D(N_i)$ that may occur in covers of $V(Z)\cap D(N)$
the \emph{ramifications} of the cover.

The following examples illustrate the PLS definition and also sketch
the prior approach to PLS given by \cite{Sit92}.%

\begin{example}
  Let $\F$ be a field, $Y=(y)$, and consider the PLS instance with
  \[
    \mat{A}=\begin{bmatrix} x & 0 \\ 0 & 0 \end{bmatrix},\qquad
    \vec{b}=\begin{bmatrix} 0 \\ y \end{bmatrix},
  \]
  and with empty initial constraints $N=Z=\emptyset$.  If $y\neq 0$
  then the second equation reads $0=y$ and the system is inconsistent;
  thus we have inconsistency at the points $D(y)$ (i.e., we have
  inconsistency regimes $(\bot,\{y\},\emptyset)$.  On $V(y)\cap D(x)$
  the matrix $\mat{A}$ has rank $1$ and the solution set is
  $\{(0,t)^T\mid t\in\overline{\F}\}$, so one solution regime is
  $(\vec{u}=\vec{0},\,\mat{B}=[0,1]^T,\,N'=\{x\},\,Z'=\{y\})$.  On
  $V(y,x)$ the matrix $\mat{A}$ is zero and every vector solves; one
  may take
  $(\vec{u}=\vec{0},\,\mat{B}=\mat{I}_2,\,N'=\emptyset,\,Z'=\{y,x\})$.
\end{example}

We now sketch the minor-based regime construction from \cite{Sit92}.
If, for $\mat{M}$ of size $r\times r~$, $\mat{A}$ is
$\begin{bmatrix}
  \mat{M} & \mat{B}\\
  \mat{C} & \mat{D}
\end{bmatrix}$,
and conformally
$\vec{b} = \begin{bmatrix} 
  \vec{c} & \vec{d} 
\end{bmatrix}^T$,
then a solution
$\mat{u} =\begin{bmatrix} 
  \vec{v} & \vec{w} 
\end{bmatrix}^T$
satisfies 
\begin{equation}
  \label{eq1}
  \mat{M}\vec{v} + \mat{B}\vec{w} = \vec{c}
\end{equation}
and
\begin{equation}
  \label{eq2}
  \mat{C}\vec{v} + \mat{D}\vec{w} = \vec{d} \>.
\end{equation} 
Under the condition that $\det(\mat{M})$ is nonzero and all larger
minors of $\mat{A}$ are zero, equation \eqref{eq1} can be solved with
specific solution $\vec{w} = 0$ and $\vec{v} = \mat{M}^{-1}\vec{c}$.
Provided the system is consistent (equation \eqref{eq2} holds), we
have the regime
\[
  (
  \begin{bmatrix}
    \vec{v}&\vec{w}
  \end{bmatrix}^T,
  \begin{bmatrix}
    -\mat{M}^{-1}\mat{B}\\
    \mat{I}
  \end{bmatrix},N,Z),
\]
where $N = \{ \det(\mat{M})\}$ and
$Z = \{ \mbox{all } (r+1)\times (r+1) \mbox{ minors of }\mat{A}\}$.

\begin{definition}
  A solution regime obtained in this way (from a choice of nonsingular
  $r\times r$ submatrix $\mat{M}$) is called a \emph{minor-defined
    regime}.
\end{definition}

Since an $n \times n$ matrix has
$\sum_{k = 0}^n \binom{n}{k}^2 = \binom{2n}{n}$ minors, there are
exponentially many minor defined regimes.  However, some of these
regimes may not be solutions due to inconsistency or it may be
possible to combine several regimes into one.  For instance if
$\det(\mat{M})$ is a constant, and $\vec{b} = 0$, then all rank $r$
solutions are covered by this one regime. \cite{Sit92} has made a
thorough study of minor defined regimes and their simplifications.

Another approach is to base solution regimes on the pivot choices in
an LU decomposition.  The simplest thing to do is to leave it to the
user, although one has to also inform the user through a proviso when
this might be necessary~\citep{CorJef97}.  That is, provide the
generic answer, but also provide a description of the set $N$.  A more
sophisticated approach is developed by~\cite{CorMor17,CamMor20} using
the theory of regular chains and its implementation in
Maple~\citep{LemMaz05} to manage the algebraic conditions.  For
example a given matrix entry may be used as a pivot, with validity
dependent on adding the polynomial to the non-zero part, $N$, of the
Zariski-constructible set.  For a comprehensive solution the case that
that entry is zero must also be pursued.  In the worst case, this
leads to a tree of zero/nonzero choices of depth $n$ and branching
factor $n$.

\section{Triangular Smith forms and degree bounds}
\label{sec:trismith}

In this paper we take a different approach, with the solution regimes
arising from Hermite normal forms, of which triangular Smith forms are
a special case.  We give a system of solution regimes of polynomial
size in the matrix dimension, $n$, and polynomial degree, $d$. Each
regime is computed in polynomial time and the regime count is
exponential only in the number of parameters.  To use Hermite forms we
will need to work over a principal ideal domain such as, for
parameters $x,y$, $\F(y)[x]$.  We will restrict our input matrix to be
polynomial in the parameters.  This first lemma shows it is not a
severe constraint.

\begin{lemma}
  Let $(\mat{A}, \vec{b},N,Z)$ be a well defined PLS over field
  $\F(Y)$, for parameter set $Y$, with $\mat{A} \in \F(Y)^{m\times n}$
  and $\vec{b} \in \F(Y)^m$ with numerator and denominator degrees
  bounded by $d$ in each parameter of $Y$ and well defined in the
  sense of Definition~\ref{def:welldefined}.  The problem is
  equivalent (same solutions) to one in which the entries of the
  matrix and vector are polynomial in the parameters $Y$, the
  dimension is the same, and the degrees are bounded by $(n+1)d$.
\end{lemma}

\begin{proof}
  Because the PLS is well defined, it is specified by $N$ that all
  denominator factors of $\mat{A}(a),\vec{b}(a)$ are nonzero for
  $a \in V(Z)\cap D(N)$.  Let $\mat{L}$ be a diagonal matrix with the
  $i$-th diagonal entry being the least common multiple ($\lcm$) of
  the denominators in row $i$ of $\mat{A}$, $\vec{b}$.  These $\lcm$s
  also evaluate to nonzero on $V(Z)\cap D(N)$. It follows that
  $L(a)\mat{A}(a)\vec{u}(a) = \mat{L}(a)\vec{b}(a)$ if and only if
  $\mat{A}(a)\vec{u}(a) = \vec{b}(a)$. Thus the PLS
  $(\mat{L}\mat{A}, \mat{L}\vec{b}, N,Z)$ is equivalent and its
  matrix and vector have polynomial entries of degrees bounded by
  $(n+1)d$.
\end{proof}
We will reduce PLS to triangular Smith normal form computations.  The
rest of this section concerns computation of triangular Smith normal
form and bounds for the degrees of the form and its unimodular
cofactor.

\begin{definition}
  Given field $\K$ and variable $x$, a matrix $\mat{H}$ over $\K[x]$
  is in (reduced) \textbf{Hermite normal form} if it is upper
  triangular, its diagonal entries are monic, and, for each column in
  which the diagonal entry is nonzero, the off-diagonal entries are of
  lower degree than the diagonal entry.

  If each diagonal entry of $\mat{H}$ exactly divides all those below
  and to the right, then $\mat{H}$ is column equivalent to a diagonal
  matrix with the same diagonal entries (its Smith normal form).  An
  equivalent condition is that, for each $i$, the greatest common
  divisor of the $i\times i$ minors in the leading $i$ columns equals
  the greatest common divisor of all $i\times i$ minors. Following
  \cite[Section 8, Definition 8.2]{Sto00} we call such a Hermite normal
  form a \textbf{triangular Smith normal form}.  It will be the
  central tool in our PLS solution.
\end{definition}

For notational simplicity, we've left out the possibility of echelon
structure in a Hermite normal form.  We will talk of Hermite normal
forms only for matrices having leading columns independent up to the
rank of the matrix.  In our algorithms we assume this hypothesis (or
enforce it by a generic right multiplication by a constant nonsingular
matrix, as in Fact~\ref{fact:rand}).  Every such matrix over $\K[x]$
is row equivalent to a unique matrix in Hermite form as defined above.
For given $\mat{A}$ we have $\mat{U}\mat{A} = \mat{H}$, with $\mat{U}$
unimodular, i.e.  $\det(\mat{U}) \in \K^*$, and $\mat{H}$ in Hermite
form.  If $\mat{A}$ is nonsingular, the unimodular cofactor $\mat{U}$
is unique and has determinant $1/c$, where $c$ is the leading
coefficient of $\det(\mat{A})$.  This follows since
$\det(\mat{U})\det(\mat{A}) = \det(\mat{H})$, which is monic.

The next definition and lemma concern assurance that Hermite form
computation will yield a triangular Smith form.

\begin{definition}
  Call a matrix \textbf{well-tempered} %
  if its Hermite form is a triangular Smith form (each diagonal entry
  exactly divides those below and to the right).  In particular, a
  well-tempered matrix has leading columns independent up to the rank.

  %
\end{definition}

There is always a column transform (unimodular matrix $\mat{R}$
applied from the right) such that $\mat{A}\mat{R}$ is
well-tempered. The following fact, proven by \cite{KalKri87} shows
that a random transform over $\F$ suffices with high probability.

\begin{fact}
  \label{fact:rand}
  Let $\mat{A}$ be a $m\times n$ matrix over $\K[x]$ of degree in $x$
  at most $d$.  Let $\mat{R}$ be a unit lower triangular matrix with
  below diagonal elements chosen from subset $S$ of $\K$ uniformly at
  random.  Then $\mat{A}\mat{R}$ is well-tempered over $\K[x]$ with
  probability at least $1 - 4n^3d/|S|$.

  Note that $\deg_x(\mat{A}\mat{R}) = \deg_x(\mat{A})$ and, for
  $\K = \F(y), \mat{A} \in \F[y,x]^{m\times n}$ and $S \subseteq \F$ we
  also have $\deg_y(\mat{A}\mat{R}) = \deg_y(\mat{A})$.
\end{fact} 

We continue with analysis of degree bounds for Hermite forms of
matrices, particularly degree bounds for triangular Smith forms of
well-tempered matrices.  The first result needed is the following fact
from~\citep{GieKim13}.  Through the remainder of this paper we will
employ ``soft O" notation, where, for functions $f,g\in\RR^k\to\RR$ we
write $f=\softO(g)$ if and only if $f=O(g\cdot\log^c |g|)$ for some
constant $c>0$.

\begin{fact}
  \label{fact:ore}
  Let $\F$ be a field, $x, y$ parameters, and let $\mat{A}$ be in
  $\F[y,x]^{n\times n}$, nonsingular, with $\deg_x(\mat{A}) \leq d$,
  $\deg_y(\mat{A}) \leq e$.  Over $\F(y)[x]$, let $\mat{H}$ be the unique
  Hermite form row equivalent to $\mat{A}$ and $\mat{U}$ be the unique
  unimodular cofactor such that $\mat{U}\mat{A} = \mat{H}$. The
  coefficients of the entries of $\mat{H}$, $\mat{U}$ are rational
  functions of $y$. Let $\Delta$ be the least common multiple of the
  denominators of the coefficients in $\mat{H}$, $\mat{U}$, as
  expressed in lowest terms.
\begin{enumerate}[label=(\alph*)]
\item $\deg_x(\mat{U}) \leq (n-1)d$ and $\deg_x(\mat{H}) \leq nd.$
\item $\deg_y(\num(\mat{H})), \deg_y(\num(\mat{U})) \leq n^2de$
  (bounds both numerator and denominator degrees).
\item $\deg_y(\Delta) \leq n^2de$.
\item $\mat{H}$ and $\mat{U}$ can be computed in polynomial time:
  deterministically in $\softO (n^9d^4e)$ time and Las Vegas
  probabilistically (never returns incorrect result) in
  $\softO (n^7d^3e)$ expected time.
\end{enumerate}
\end{fact}

\begin{proof}
  This is~\citep[Summary Theorem]{GieKim13}.
  The situation there is more abstract, more involved.  We offer this
  tip to the reader: their $\partial, z, \sigma, \delta$ correspond
  respectively to our $x, y$, identity, zero.

  Item (c) is not stated explicitly in a theorem of~\citep{GieKim13}
  but is evident from the proofs of Theorems 5.2 and 5.6 there. The
  common denominator is the determinant of a matrix over $\K[z]$ of
  dimension $n^2d$ and with entries of degree in $z$ at most $e$.
\end{proof}

We will generalize this fact to nonsingular and non-square matrices in
Theorem \ref{thm:bounds}.  In that case the unimodular cofactor,
$\mat{U}$, is not unique and may have arbitrarily large degree
entries.  The following algorithm is designed to produce a $\mat{U}$
with bounded degrees.

\begin{algorithm}[H]
  \caption{\texttt{\mat{U}, \mat{H} = HermiteForm(\mat{A})}}
  \label{alg:HF}
  \begin{algorithmic}[1]

    \REQUIRE  Well-tempered matrix $\mat{A} \in \F[y, x]^{m\times n}$, for field $\F$ and parameters $x, y$.\\

    \ENSURE For $\K = \F(y)$, Unimodular $\mat{U} \in \K[x]^{m\times m}$ and $\mat{H} \in \K[x]^{m\times n}$ in triangular Smith form such that $\mat{U}\mat{A} = \mat{H}$.

    The point of the specific method given here is to be able, in
    Theorem \ref{thm:bounds} below, to bound
    $\deg_x(\mat{U}, \mat{H})$ and $\deg_y(\mat{U}, \mat{H})$
    (numerators and denominators).

    \STATE Compute $r = \rank(\mat{A})$ and nonsingular $\mat{U}_0 \in \K^{m\times m}$ such that $\mat{U}_0 \mat{A}$ has nonsingular leading $r\times r$ minor. 
Because $\mat{A}$ is well-tempered the first $r$ columns are independent and such $\mat{U}_0$ exists.  
$\mat{U}_0$ could be a permutation found via Gaussian elimination, say, or a random unit upper triangular matrix.  In the random case, failure to achieve nonsingular leading minor becomes evident in the next step, so that the randomization is Las Vegas.

\STATE Let $\mat{U}_0 \mat{A} =
\begin{bmatrix} \mat{A}_1 & \mat{A}_2 \\ \mat{A}_3 & \mat{A}_4 \end{bmatrix}$ and
$\mat{B} =
\begin{bmatrix} \mat{A}_1 & \mat{0}_{r\times m-r} \\ \mat{A}_3 &
  \mat{I}_{m-r} \end{bmatrix}$.  $\mat{B}$ is nonsingular. Compute its
unique unimodular cofactor $\mat{U}_1$ and Hermite form
$\mat{T} = \mat{U}_1\mat{B} =
\begin{bmatrix} \mat{H}_1 & * \\ 0 & * \end{bmatrix}$. 

\STATE Let
$\mat{H} = \mat{U}_1\mat{U}_0\mat{A} =
\begin{bmatrix} \mat{H}_1 & \mat{H}_2 \\ 0 & 0 \end{bmatrix}$.
Since the input is assumed well-tempered, the proof of
Theorem~\ref{thm:bounds} shows that $\mat{H}$ is in triangular Smith
form.

\STATE Let $\mat{U} = \mat{U}_1\mat{U}_0$ and return $\mat{U}$, $\mat{H}$.
\end{algorithmic}
\end{algorithm}

\begin{theorem}
  \label{thm:bounds}
  Let $\F$ be a field, $x, y$ parameters, and let $\mat{A}$ be in
  $\F[y,x]^{m\times n}$ of rank $r$, $\deg_x(\mat{A}) \leq d$, and
  $\deg_y(\mat{A}) \leq e$.  Let $\mat{R}$ be a random unit lower triangular
  matrix chosen as in Fact~\ref{fact:rand}, so that $\mat{AR}$ is
  well-tempered with the stated probability.  Then, for the triangular
  Smith form $\mat{U}\mat{A}\mat{R} = \mat{H}$ computed as $\mat{U}$,
  $\mat{H}$ = {\normalfont \texttt{HermiteForm}}($\mat{A}\mat{R}$), we
  have
  \begin{enumerate}[label=(\alph*)]
  \item Algorithm {\normalfont \texttt{HermiteForm}} is (Las Vegas)
    correct and runs in expected time $\bigO(m^7d^3e)$;
  \item $\deg_x(\mat{U}, \mat{H}) \leq md$;
  \item $\deg_y(\mat{U}, \mat{H}) = \softO(m^2de)$.
  \end{enumerate}
\end{theorem}

\begin{proof}
  Let $\mat{R}$ be as in Fact \ref{fact:rand} with $\K = \F(y)$ and
  $S \subseteq \F$.  If the field $\F$ is small, an extension field
  can be used to provide large enough $S$.

  We apply \texttt{HermiteForm} to $\mat{A}\mat{R}$ to obtain
  $\mat{U}$, $\mat{H}$, and use the notation of Algorithm~\ref{alg:HF}
  in this proof.  We see by construction that $\mat{B}$ is
  nonsingular, from which it follows that $\mat{U}_1$ and $\mat{T}$
  are uniquely determined (since $\mat{T}$ is the unique Hermite form
  of $\mat{B}$).

  By construction the first $r$ columns of $\mat{B}$ equal the first
  $r$ columns of $\mat{U}_0\mat{A}\mat{R}$, hence the first $r$
  columns of $\mat{T}=\mat{U}_1\mat{B}$ equal the first $r$ columns of
  $\mat{H}:=\mat{U}_1\mat{U}_0\mat{A}\mat{R}$.  Since $\mat{T}$ is
  upper triangular, all entries of $\mat{T}$ below row $r$ in its
  first $r$ columns are zero, and therefore the same holds for
  $\mat{H}$.  Because $\rank(\mat{A}\mat{R})=r$ and the leading
  $r\times r$ block $\mat{H}_1$ is nonsingular, the top $r$ rows of
  $\mat{H}$ are linearly independent and span the row space of
  $\mat{H}$.  Any linear combination of these rows that is zero in the
  first $r$ columns must therefore be trivial (since the restriction
  to the first $r$ columns has nonsingular coefficient matrix
  $\mat{H}_1$).  It follows that the last $m-r$ rows of $\mat{H}$ are
  zero, so $\mat{H}$ has the block form
  $\begin{bmatrix} \mat{H}_1 & \mat{H}_2 \\ 0 & 0 \end{bmatrix}$.

  Moreover, the first $r$ columns of $\mat{H}$ coincide with those of
  the Hermite form $\mat{T}$ of $\mat{B}$, so they satisfy the reduced
  degree conditions required in Hermite normal form.  Thus $\mat{H}$
  is in Hermite normal form and is row equivalent to $\mat{A}\mat{R}$.
  By uniqueness of Hermite normal form for matrices whose leading
  columns are independent up to rank, $\mat{H}$ is the Hermite form of
  $\mat{A}\mat{R}$.  Since $\mat{A}\mat{R}$ is well-tempered
  (Fact~\ref{fact:rand}), this Hermite form is a triangular Smith
  form, as required.  The runtime is dominated by computation of
  $\mat{U}_1$ and $\mat{T}$ for $\mat{B}$, so Fact \ref{fact:ore}
  provides the bound in (a).

  For the degree in $x$, applying Fact \ref{fact:ore}, we have
  $\deg_x(\mat{U}_1) \leq (m-1)d$.  Noting that $\mat{U}_0$ has degree
  zero, we have $\deg_x(\mat{U}) = \deg_x(\mat{U}_1)$ and
  $\deg_x(\mat{H}) = \deg_x(\mat{U}) + \deg_x(\mat{A}) \leq (m-1)d+d =
  md$.

  For the degree in $y$, note first that the bounds $d,e$ for degrees
  in $\mat{A}$ apply as well to $\mat{B}$.  We have, by Fact
  \ref{fact:ore}, that $\deg_y(\num(\mat{U}_1)) = \softO(m^2de)$ and
  the same bound for $\deg_y(\den(\mat{U}_1))$.  For $\mat{H}$, note
  that
  $\num(\mat{H})/\den(\mat{H}) = \num(\mat{U})\mat{A}/\den(\mat{U})$
  so that
  $\deg_y(\den(\mat{H})) \leq \deg_y(\mat{U}) = \softO(m^2de)$, and
  $\deg_y(\num(\mat{H})) \leq \deg_y(\num(\mat{U})A) = \softO(m^2de) +
  e = \softO(m^2de)$.
\end{proof}

\section{Reduction of PLS to triangular Smith forms}
\label{sec:pls2snf}

In this section we define the Comprehensive Triangular Smith Normal
form problem and solution and show that PLS can be reduced to it.  The
next section addresses the solution of CTSNF itself.
\begin{definition}
  \label{def:paramext}
  For field $\F$ and parameters $Y = (y_1, \ldots, y_s)$, $\F_Y$ is
  a \emph{parameterized extension} of $\F$ if $\F_Y$ is the top of a
  tower obtained by adjoining the parameters of $Y$ in some order,
  each either as a rational function parameter or as algebraic over
  the field constructed so far.  Thus, at an algebraic step adjoining
  $y_i$, the field is extended by $\F_{j-1}[y_i]/\langle f_i \rangle$,
  where $f_i$ is irreducible over $\F_{j-1}$ as a polynomial in
  $y_i$.  When a solution regime to a PLS or CTSNF problem is over a
  parameterized extension $\F_Y$, we record the irreducible
  polynomials defining the algebraic steps of the tower in the
  vanishing constraint set $Z'$ of that regime.

  A comprehensive triangular Smith normal form problem (\emph{CTSNF
    problem}) is a triple $(\mat{A}, N, Z)$ of a matrix $\mat{A}$ over
  $\F[Y,x]$ together with polynomial sets $N,Z \subseteq \F[Y]$
  defining a Zariski-constructible constraint on the parameters $Y$.
  Here $x$ is treated as a distinguished variable: in CTSNF regimes we
  do not specialize $x$, and evaluation at $a\in V(Z)\cap D(N)$ refers
  to the partial evaluation $\mat{A}(a)\in \F[x]^{m\times n}$.

  For CTSNF problem $(\mat{A}, N, Z)$, a \emph{triangular Smith
    regime} is of the form $(\mat{U}, \mat{H}, \mat{R}, N', Z')$, with
  $\mat{U}$, $\mat{H}$ over $\F_Y[x]$, where $\F_Y$ is a parameterized
  extension of $\F$ and any polynomials defining algebraic extensions
  in the tower are in $Z'$, and with $N',Z'\subseteq \F[Y]$, such that
  all denominator factors occurring in $\mat{U}$ and $\mat{H}$ are
  contained in $N'$, and on all $a \in V(Z')\cap D(N')$, $\mat{H}(a)$
  is in triangular Smith form over $\F(a)[x]$, $\mat{U}(a)$ is
  unimodular in $x$, $\mat{R}$ is nonsingular over $\F$, and
  $\mat{U}(a)\mat{A}(a)\mat{R} = \mat{H}(a)$.

  A \emph{CTSNF solution} is a list
  $\{(\mat{U}_i, \mat{H}_i, \mat{R}_i, N_i, Z_i)| i \in 1,
  \ldots,k\}$, of \emph{triangular Smith regimes} that cover
  $V(Z)\cap D(N)$, which is to say
  $V(Z)\cap D(N) \subseteq \cup\{V(Z_i)\cap D(N_i)|i \in
  1,\ldots,k\}$.
\end{definition}

The goal in this section is to reduce the PLS problem to the CTSNF
problem. The first step is to show it suffices to consider PLS with a
matrix already in triangular Smith form.  The second step is to show
each CTSNF solution regime generates a set of PLS solution regimes.

\begin{lemma}
  \label{lem:equiv}
  Given a parameterized field $\F_Y$ and matrix $\mat{A}$ over
  $\F[Y,x]$, let $\mat{H}$ be a triangular Smith form of $\mat{A}$
  over $\F_Y[x]$, with $\mat{U}$ unimodular over $\F_Y[x]$, and
  $\mat{R}$ nonsingular over $\F$ such that $\mat{U}\mat{A}\mat{R}$ =
  $\mat{H}$.  Let $N',Z'$ be constraints such that all data below are
  defined and $\mat{U}(a)$ is unimodular for every
  $a\in V(Z')\cap D(N')$.  Then $(\vec{u}, \mat{B}, N', Z')$ is a
  solution regime for PLS problem $(\mat{A}, \vec{b}, N,Z)$ over
  $\F[Y,x]$ if and only if
  $(\mat{R}^{-1}\vec{u}, \mat{R}^{-1}\mat{B}, N', Z')$ is a
  solution regime for PLS problem $(\mat{H}$, $\mat{U}\vec{b}, N,Z)$.
  Moreover, $(N',Z')$ is an inconsistency regime for one system if and
  only if it is an inconsistency regime for the other.
\end{lemma}
\begin{proof}
  Let $(a,\xi)\in V(Z')\cap D(N')$.  By hypothesis, all expressions
  appearing below are well defined and $\mat{U}(a)$ is unimodular
  over $\F(a)[x]$.  Hence $\mat{U}(a,\xi)$ is invertible.  The
  matrix $\mat{R}$ is constant and nonsingular over $\F$.  Thus the
  following are equivalent.
  \begin{enumerate}
  \item $\mat{A}(a,\xi)\vec{u}(a,\xi) = \vec{b}(a,\xi).$
  \item $\mat{U}(a,\xi)\mat{A}(a,\xi)\vec{u}(a,\xi) =
    \mat{U}(a,\xi)\vec{b}(a,\xi).$
  \item
    $(\mat{U}(a,\xi)\mat{A}(a,\xi)\mat{R})(\mat{R}^{-1}\vec{u}(a,\xi)) =
    \mat{U}(a,\xi)\vec{b}(a,\xi).$
  \end{enumerate}
  The same equivalence shows that
  $\mat{A}(a,\xi)\vec{u}=\vec{b}(a,\xi)$ is consistent if and only if
  $\mat{H}(a,\xi)\vec{w}=\mat{U}(a,\xi)\vec{b}(a,\xi)$ is consistent,
  so inconsistency regimes are preserved as well.
\end{proof}

Then we have the following algorithm to solve a PLS with the matrix
already in triangular Smith form.  For simplicity we assume a square
matrix, the rectangular case being a straightforward extension.

\begin{algorithm}[h]
  \caption{\texttt{TriangularSmithPLS}}
  \label{alg:TrianglarSmithPLS}
  \begin{algorithmic}[1]
    \REQUIRE A PLS problem $(\mat{H},\vec{b},N,Z)$ with $\mat{H} \in \F_Y[x]^{n\times n}$ and $\vec{b} \in \F_Y[x]^n$, where $\F_Y$ is a parameterized extension for parameter list $Y$ and $x$ is the distinguished variable, with $N,Z \subseteq \F[Y,x]$ and $\mat{H}$ in triangular Smith form.
    \ENSURE $S$, a corresponding list of regimes (solution regimes and inconsistency regimes).
    \STATE For a polynomial $s(x)$ let $\sqfr(s)$ denote the square-free part in $x$.
    \STATE Let $s_i$ denote the $i$-th diagonal entry of $\mat{H}$ and define $s_0 = 1$.  Let
    $r_0=\max\{i\in\{0,\ldots,n\}\mid s_i\neq 0\}$ (so $s_{r_0+1}=\cdots=s_n=0$), and set $s_{r_0+1}=0$.
    \STATE For $i \in 0, \ldots, r_0-1$, define $f_i = \sqfr(s_{i+1})/\sqfr(s_i)$, and set $f_{r_0}=0$.
    Let $\mathcal{I}$ be the set of indices $i\in\{0,\ldots,r_0\}$ such that $f_i$ has positive degree in $x$ or is zero.
    \STATE All adjunctions to $N$ and $Z$ below involving elements of $\F_Y(x)$ are interpreted using the numerator/denominator convention of Section~\ref{sec:defs}.
    \STATE Initialize $S \gets \emptyset$.
    \FOR{$r \in \mathcal{I}$}
    \STATE Let $\mat{H}_r$ be the leading $r\times r$ submatrix of $\mat{H}$, $\vec{b}_r = (b_1, \ldots, b_r)^T$, and let $\mat{H}_{12}$ be the $r\times (n-r)$ block of the first $r$ rows and last $n-r$ columns of $\mat{H}$.
    \STATE Let $\vec{u} =(\mat{H}_r^{-1}\vec{b}_r, 0_{n-r})$.
    \STATE Let $\mat{B} = \begin{bmatrix}-\mat{H}_r^{-1}\mat{H}_{12}\\ \mat{I}_{n-r}\end{bmatrix}$.
    \STATE Define $N_r = N \cup \{s_r,\,\den(\vec{u}),\,\den(\mat{B})\}$ and $Z_r = Z \cup \{f_r,\, b_{r+1}, \ldots, b_n\}$.
    \STATE Adjoin the solution regime $(\vec{u}, \mat{B}, N_r, Z_r)$ to $S$.
    \FOR{$j=r+1,\ldots,n$}
      \STATE Define $N_{r,j}= N\cup \{s_r,\, b_j\}$ and $Z_{r,j}= Z\cup \{f_r\}$.
      \STATE Adjoin the inconsistency regime $(\bot,N_{r,j},Z_{r,j})$ to $S$.
    \ENDFOR
    \ENDFOR
    \STATE Return $S$.
  \end{algorithmic}
\end{algorithm}

\begin{lemma}
  \label{lem:trismithpls}
  Algorithm {\normalfont {\texttt TriangularSmithPLS}} is correct: it
  outputs a list of solution regimes and inconsistency regimes whose
  constructible sets cover $V(Z)\cap D(N)$.  Let
  $r_0=\max\{i\mid s_i\neq 0\}$ where $s_i$ are the diagonal entries
  of $\mat{H}$ (so $\mat{H}$ has rank $r_0$ over $\F_Y(x)$), and let
  $d=\sum_{i=1}^{r_0}\deg_x(s_i)$.  Then the algorithm outputs at most
  $1+\sqrt{2d}$ solution regimes, and at most $(n+1)(1+\sqrt{2d})$
  regimes total.
\end{lemma}
\begin{proof}
  For $r\in\mathcal{I}$, the algorithm outputs a solution regime
  $(\vec{u},\mat{B},N_r,Z_r)$.  Using the numerator/denominator
  convention of Section~\ref{sec:defs}, $s_r\in N_r$ imposes
  $s_r\neq 0$ and $f_r\in Z_r$ imposes $f_r=0$ wherever the relevant
  rational functions are defined.  Thus for any
  $(a,\xi)\in V(Z_r)\cap D(N_r)$ we have
  $s_r(a,\xi)\neq 0$.  If $r<r_0$ then
  $f_r=\sqfr(s_{r+1})/\sqfr(s_r)$, so $f_r(a,\xi)=0$ implies
  $s_{r+1}(a,\xi)=0$ while $\sqfr(s_r)(a,\xi)\neq 0$, hence
  $s_r(a,\xi)\neq 0$.  If $r=r_0$ then $s_{r+1}$ is identically zero,
  so $s_{r+1}(a,\xi)=0$ holds automatically.  In either case, the
  divisibility properties of triangular Smith form imply that all
  entries in rows $r+1,\ldots,n$ of $\mat{H}(a,\xi)$ are zero.

  Since $s_1\mid s_2\mid\cdots\mid s_r$ in a triangular Smith form, the
  conditions above also imply $s_1(a,\xi),\ldots,s_r(a,\xi)\neq 0$.
  Therefore the leading block $\mat{H}_r(a,\xi)$ is invertible.  The
  regime adjoins $b_{r+1},\ldots,b_n$ to $Z_r$, so for
  $(a,\xi)\in V(Z_r)\cap D(N_r)$ the last $n-r$ equations reduce to
  $0=b_{r+1}(a,\xi),\ldots,0=b_n(a,\xi)$ and are consistent.  The
  vector $\vec{u}(a,\xi)$ satisfies the leading $r$ equations by
  construction, and the additional condition $\den(\vec{u})(a,\xi)\neq
  0$ (enforced by $\den(\vec{u})\in N_r$) guarantees that $\vec{u}$ is
  well defined.

  For the nullspace, write
  $\mat{H}=\begin{bmatrix}\mat{H}_r & \mat{H}_{12}\\ 0 &
    0\end{bmatrix}$ after evaluation at $(a,\xi)$.  Then
  \[
    \mat{H}(a,\xi)\mat{B}(a,\xi)=
    \begin{bmatrix}\mat{H}_r & \mat{H}_{12}\\ 0 & 0\end{bmatrix}
    \begin{bmatrix}-\mat{H}_r^{-1}\mat{H}_{12}\\ \mat{I}_{n-r}\end{bmatrix}
    =
    \begin{bmatrix}0\\ 0\end{bmatrix},
  \]
  and the bottom block $\mat{I}_{n-r}$ shows that the columns of
  $\mat{B}(a,\xi)$ are linearly independent.  The additional condition
  $\den(\mat{B})(a,\xi)\neq 0$ guarantees that $\mat{B}(a,\xi)$ is
  well defined.  Thus $(\vec{u},\mat{B},N_r,Z_r)$ is a correct
  solution regime.

  The algorithm also outputs inconsistency regimes
  $(\bot,N_{r,j},Z_{r,j})$ for $j=r+1,\ldots,n$.  For any
  $(a,\xi)\in V(Z_{r,j})\cap D(N_{r,j})$ we have the same rank
  conditions as above (since $s_r\in N_{r,j}$ and $f_r\in Z_{r,j}$),
  so the last $n-r$ rows of $\mat{H}(a,\xi)$ are zero, while
  $b_j(a,\xi)\neq 0$ (since $b_j\in N_{r,j}$).  Hence
  $\mat{H}(a,\xi)\vec{u}=\vec{b}(a,\xi)$ is inconsistent, so
  $(\bot,N_{r,j},Z_{r,j})$ is a correct inconsistency regime.

  To see that these regimes cover $V(Z)\cap D(N)$, take
  $(a,\xi)\in V(Z)\cap D(N)$ and let $r=\rank(\mat{H}(a,\xi))$.  By
  the divisibility chain on the diagonal entries, $s_r(a,\xi)\neq 0$
  and $s_{r+1}(a,\xi)=0$, which implies either $r=r_0$ or
  $f_r(a,\xi)=0$.  Thus $r\in\mathcal{I}$ and the algorithm generates
  the corresponding regimes.  If
  $\mat{H}(a,\xi)\vec{u}=\vec{b}(a,\xi)$ is consistent, then
  necessarily $b_{r+1}(a,\xi)=\cdots=b_n(a,\xi)=0$, so
  $(a,\xi)\in V(Z_r)\cap D(N_r)$ and the solution regime applies.  If
  it is inconsistent, then there exists $j>r$ with $b_j(a,\xi)\neq 0$,
  so $(a,\xi)\in V(Z_{r,j})\cap D(N_{r,j})$ and an inconsistency
  regime applies.

  For the regime count, let $r_1<\cdots<r_t$ be the indices with
  $\deg_x(f_{r_j})>0$.  Each $f_{r_j}$ has a square-free irreducible
  factor of degree at least $1$ that divides
  $s_{r_j+1},\ldots,s_{r_0}$, hence contributes at least $r_0-r_j$ to
  $d=\sum_{i=1}^{r_0} \deg_x(s_i)$.  Therefore
  \[
    d \ \geq\ \sum_{j=1}^t (r_0-r_j)\ \geq\ 1+2+\cdots+t\ =\ t(t+1)/2,
  \]
  so $t\leq \sqrt{2d}$.  Since the algorithm outputs at most one
  additional solution regime coming from $f_{r_0}=0$, it outputs at
  most $t+1\leq 1+\sqrt{2d}$ solution regimes.  Each such $r$ produces
  at most $n-r$ inconsistency regimes, so the total number of regimes
  is at most $(n+1)(t+1)\leq (n+1)(1+\sqrt{2d})$.
\end{proof}

\begin{algorithm}[H]
\caption{\texttt{PLSviaCTSNF}}
\label{alg:PLSviaCTSNF}
\begin{algorithmic}[1]
  \REQUIRE A PLS problem $(\mat{A}, \vec{b},N,Z)$ over $\F[Y,x]$, for parameter list $Y$ and distinguished variable $x$.
  \ENSURE A corresponding PLS solution $S$ consisting of solution regimes and inconsistency regimes.
  \STATE Split the constraint sets as $N=N_Y\cup N_x$ and $Z=Z_Y\cup Z_x$, where $N_Y=N\cap \F[Y]$, $Z_Y=Z\cap \F[Y]$ and $N_x=N\setminus N_Y$, $Z_x=Z\setminus Z_Y$.
  \STATE Over the ring $\F(Y)[x]$, let $T$ solve the CTSNF problem $(\mat{A}, N_Y,Z_Y)$. $T$ is a set of triangular Smith regimes of form $(\mat{U}, \mat{H},\mat{R}, N',Z')$. Let $S = \emptyset$.
  \STATE For each triangular Smith regime $(\mat{U}, \mat{H},\mat{R},N',Z')$ in $T$, using algorithm
  \texttt{TriangularSmithPLS} solve the PLS problem $(\mat{H}$, $\mat{U}\vec{b},\, N'\cup N_x,\, Z'\cup Z_x)$.  For each returned regime, adjoin it to $S$, adjusting solution data by factor $\mat{R}$ as in Lemma~\ref{lem:equiv} (and leaving inconsistency regimes unchanged).
  \STATE Return $S$.
\end{algorithmic}
\end{algorithm}

\begin{theorem}
  Algorithm \ref{alg:PLSviaCTSNF} is correct.
\end{theorem}
\begin{proof}
  Let $(a,\xi)\in V(Z)\cap D(N)$.  By construction $(a,\xi)$ satisfies
  the $Y$-only constraints $V(Z_Y)\cap D(N_Y)$, so at least one
  triangular Smith regime of $T$ in step $2$ is valid at $a$.  For
  that regime, step $3$ applies Algorithm~\ref{alg:TrianglarSmithPLS}
  and produces either a solution regime or an inconsistency regime
  whose constructible set contains $(a,\xi)$
  (Lemma~\ref{lem:trismithpls}).  Lemma~\ref{lem:equiv} transfers
  solution data back through $\mat{R}$ without changing the underlying
  constraint set, so the corresponding regime for the original PLS
  problem is correct and covers~$(a,\xi)$.  
\end{proof}

\section{Solving Comprehensive Triangular Smith Normal Form}
\label{sec:solve}

In view of the reductions of the preceding section, to solve a
parametric linear system it remains only to solve a comprehensive
triangular Smith form problem.  This is difficult in general but we
give a method to give a comprehensive solution with polynomially many
regimes in the bivariate and trivariate cases.


\begin{theorem}
  \label{thm:bivarctsnf}
  Let $\rho=\max(m,n)$, and let $\mat{A}\in \F[y,x]^{m\times n}$
  have degree $d$ in $x$ and degree $e$ in $y$.  Let $N,Z$ be
  polynomial sets defining a Zariski-constructible constraint on $y$.
  Then the CTSNF problem $(\mat{A},N,Z)$ has a solution of at most
  $\softO(\rho^2de)$ triangular Smith regimes.
\end{theorem}
\begin{proof}
  We first solve the unconstrained case $N=Z=\emptyset$.  Compute a
  triangular Smith form $\mat{U}_0, \mat{H}_0, \mat{R}_0$ over
  $\F(y)[x]$ such that $\mat{U}_0\mat{A}\mat{R}_0=\mat{H}_0$.  This
  regime is valid for evaluations that do not zero the denominators
  (polynomials in $y$) of $\mat{H}_0$ and $\mat{U}_0$.  Let
  $\Delta_0=\den(\mat{U}_0,\mat{H}_0)$ and set
  $N_0=\mathrm{Fac}(\Delta_0)$, with $Z_0=\emptyset$, to complete
  the generic regime.

  Then, for each $f\in N_0$, adjoin the regime
  $(\mat{U}_f,\mat{H}_f,\mat{R}_f,
  N_f=N_0\setminus\{f\}, Z_f=\{f\})$, obtained by computing the
  triangular Smith form over $(\F[y]/\langle f\rangle)[x]$.
  Together with the generic regime indexed by $0$, these regimes cover
  all specializations of $y$: at any point either no factor of
  $\Delta_0$ vanishes, or at least one irreducible factor $f$ of
  $\Delta_0$ vanishes.  From the bounds of Theorem~\ref{thm:bounds},
  $\deg_y(\Delta_0)=\softO(\rho^2de)$, so the number of irreducible
  factors, and hence the number of regimes, is $\softO(\rho^2de)$.

  For general input constraints $N,Z$, intersect each produced regime
  with $V(Z)\cap D(N)$, i.e., replace each output
  $(\mat{U}_*,\mat{H}_*,\mat{R}_*,N_*,Z_*)$ by
  $(\mat{U}_*,\mat{H}_*,\mat{R}_*,N\cup N_*,Z\cup Z_*)$.  This
  changes neither the validity of the regimes nor their number.
\end{proof}

We can proceed in a similar way when there are three parameters, but
must address an additional complication that arises.

Before proving the trivariate bound, we use the following degree-bound
observation.  The denominator degree bounds used in
Theorem~\ref{thm:bounds} apply, with the same proof, when the single
coefficient parameter $y$ is replaced by a pair of coefficient
parameters $y,z$.  In particular, for
$\mat{A}\in \F[y,z,x]^{m\times n}$ with degree at most $d$ in $x$
and at most $e$ in each of $y,z$, the common denominator of the
computed generic triangular Smith regime over $\F(y,z)[x]$ has degree
$\softO(\rho^2de)$ in each of $y$ and $z$.  The same bound applies to
the one-parameter denominators arising after restricting to an
irreducible curve factor and computing over the corresponding
parameterized coefficient field.  This follows from the determinant
expressions controlling the common denominators in Fact~\ref{fact:ore}
and Theorem~\ref{thm:bounds}, with the remaining coefficient
parameters carried through in the coefficient field.

\begin{theorem}
  \label{thm:trivarctsnf}
  Let $\rho=\max(m,n)$, and let $\mat{A}\in \F[z,y,x]^{m\times n}$
  have degree $d$ in $x$ and degree $e$ in $y,z$.  Let $N,Z$ be
  polynomial sets defining a Zariski-constructible constraint on $y$
  and $z$. Then the CTSNF problem $(\mat{A},N,Z)$ has a solution of
  at most $\softO(\rho^4d^2e^2)$ triangular Smith regimes.
\end{theorem}

\begin{proof}
  As in the bivariate case above, we solve the unconstrained case and
  just adjoin $N,Z$, if nontrivial, to the Zariski-constructible
  condition of each solution regime.

  First compute a triangular Smith form
  $\mat{U}_0, \mat{H}_0, \mat{R}_0$ over $\F(y,z)[x]$ such that
  $\mat{U}_0\mat{A}\mat{R}_0=\mat{H}_0$.  This will be valid for
  evaluations that don't zero the denominators (polynomials in $y,z$)
  of $\mat{H}_0$ and $\mat{U}_0$.  Let
  $\Delta_0=\den(\mat{U}_0,\mat{H}_0)$ and set
  $N_0=\mathrm{Fac}(\Delta_0)$, and set $Z_0=\emptyset$ to complete
  the first regime.

  Next, for each $f\in N_0$, if $y$ occurs in $f$ compute a triangular
  Smith form $(\mat{U}_f,\mat{H}_f,\mat{R}_f)$ over
  $(\F(z)[y]/\langle f\rangle)[x]$.  If $y$ does not occur in $f$,
  interchange the roles of $y,z$ and compute over
  $(\F(y)[z]/\langle f\rangle)[x]$.  In either case the resulting
  coefficients lie in a parameterized extension in the sense of
  Definition~\ref{def:paramext}.

  Let $\delta_f=\den(\mat{U}_f,\mat{H}_f)$.  Here $\delta_f$ arises from
  divisions by elements of the coefficient field $\F(z)$ (respectively
  $\F(y)$), hence after clearing denominators it may be taken in
  $\F[z]$ (respectively $\F[y]$).  We therefore adjoin the regime
  \[
    (\mat{U}_f,\mat{H}_f,\mat{R}_f,\;
      N_f=(N_0\setminus\{f\})\cup \mathrm{Fac}(\delta_f),\;
      Z_f=\{f\}),
  \]
  which is valid on $V(f)\cap D(N_f)$.

  The regimes indexed by $0$ and by single factors $f\in N_0$ cover
  all specializations except for those in which either (i) at least
  two distinct factors from $N_0$ vanish simultaneously, or (ii) for
  some $f\in N_0$ the denominator $\delta_f$ vanishes on $V(f)$.  Both
  situations give rise only to zero-dimensional exceptional sets.
  Indeed, for distinct irreducible $f,g\in N_0$ the intersection
  $V(f)\cap V(g)\subseteq \overline{\F}^2$ is finite, and by
  B\'ezout's theorem~\citep{CoxLit13} it contains at most
  $\deg(f)\deg(g)$ points.  Similarly, since $f$ is irreducible and
  depends on $y$ (respectively $z$) while $\delta_f\in \F[z]$
  (respectively $\F[y]$), we have $\gcd(f,\delta_f)=1$ and again
  B\'ezout's theorem bounds $|V(f)\cap V(\delta_f)|$ by
  $\deg(f)\deg(\delta_f)$.

  These finitely many points can be enumerated (for example by a
  resultant computation or a triangular decomposition of the
  appropriate ideals).  For each such point $(y,z)\in\overline{\F}^2$
  we produce a separate regime by evaluating $\mat{A}$ at that point
  and computing a triangular Smith form over the corresponding
  (possibly algebraic) extension field, as allowed in
  Definition~\ref{def:paramext}.

  The degree bounds $\deg(\Delta_0)=\softO(\rho^2de)$ and
  $\deg(\delta_f)=\softO(\rho^2de)$ follow from the degree-bound
  observation above.  Since
  $\sum_{f\in N_0}\deg(f)=\deg(\Delta_0)$, the total number of
  point regimes arising from the intersections above is bounded by
  $\bigO(\deg(\Delta_0)^2)=\softO((\rho^2de)^2)$, yielding the stated
  $\softO(\rho^4d^2e^2)$ overall regime bound.
\end{proof}

\begin{corollary}
  For a PLS with $m\times n$ matrix $\mat{A}$, $\vec{b}$ an
  $m$-vector, let $\rho=\max(m,n)$, and suppose
  $\deg_x(\mat{A}, \vec{b}) \leq d, \deg_y(\mat{A}, \vec{b}) \leq
  e,\deg_z(\mat{A}, \vec{b}) \leq e$.  Counting both solution
  regimes and inconsistency regimes, we have
  \begin{enumerate}
  \item $\bigO(\rho^{1.5}d^{0.5})$ regimes in the PLS solution for the
    univariate case (domain of $\mat{A}$, $\vec{b}$ is $\F[x]$).
  \item $\softO(\rho^{3.5}d^{1.5}e)$ regimes in the PLS solution for the
    bivariate case (domain of $\mat{A}$, $\vec{b}$ is $\F[x,y]$).
  \item $\softO(\rho^{5.5}d^{2.5}e^2)$ regimes in the PLS solution for the
    trivariate case (domain of $\mat{A}$, $\vec{b}$ is $\F[x,y,z]$).
  \end{enumerate}
\end{corollary}
\begin{proof}
  Consider one CTSNF regime $(\mat{U},\mat{H},\mat{R},N',Z')$ for
  $\mat{A}$, so that $\mat{H}$ is in triangular Smith form with
  diagonal entries $s_i$.  Let $r_0=\max\{i\mid s_i\neq 0\}$ and
  $d_H=\sum_{i=1}^{r_0}\deg_x(s_i)$, as in
  Lemma~\ref{lem:trismithpls}.  The product of the nonzero diagonal
  entries of a triangular Smith form is the rank-$r_0$ determinantal
  divisor.  Hence it divides every nonzero $r_0\times r_0$ minor of
  the input matrix (after the constant right transformation
  $\mat{R}$).  Since some such minor has degree at most $r_0d$, we
  have $d_H\leq r_0d\leq \rho d$.

  Lemma~\ref{lem:trismithpls} therefore implies that applying
  Algorithm~{\normalfont\texttt{TriangularSmithPLS}} to this regime
  produces at most $(n+1)(1+\sqrt{2d_H})=
  \bigO(\rho^{1.5}d^{0.5})$ PLS regimes, including inconsistency
  regimes.  In the univariate case there is a single CTSNF regime.  In
  the bivariate case we multiply this per-regime bound by the
  $\softO(\rho^2de)$ CTSNF-regime bound of Theorem~\ref{thm:bivarctsnf}, and in the
  trivariate case by the $\softO(\rho^4d^2e^2)$ bound of Theorem~\ref{thm:trivarctsnf}.
  This gives the three stated estimates.
\end{proof}

\section{Normal forms and Eigenproblems}
\label{sec:nfeigs}

Comprehensive Hermite Normal form and comprehensive Smith Normal form
are immediate corollaries of our comprehensive triangular Smith form.
For Hermite form, use the same comprehensive construction but take the
right hand cofactor to be the identity, $\mat{R} = \mat{I}$, and omit
the well-tempered/triangular-Smith divisibility requirement in the
Hermite-form computation.  For Smith form one can convert
each regime of CTSNF to a Smith regime.  Where
$\mat{U}\mat{A}\mat{R} = \mat{H}$ with $\mat{H}$ a triangular
Smith form, perform column operations to obtain
$\mat{U}\mat{A}\mat{V} = \mat{S}$
with $\mat{S}$ the diagonal of $\mat{H}$.  In $\mat{H}$ the diagonal
entries divide the off diagonal entries in the same row. Subtract
multiples of the $i$-th column from the subsequent columns to
eliminate the off diagonal entries.  Because the diagonal entries are
monic, no new denominator factors arise and
$\det(\mat{V}) = \det(\mat{R}) \in \F$.  Thus when
$(\mat{U},\mat{H},\mat{R}, N,Z)$ is a valid regime in a CTSNF solution
for $\mat{A}$, then $(\mat{U},\mat{S},\mat{V}, N,Z)$ is a valid regime
for Smith normal form.

It is well known that if $\mat{A} \in \K^{n\times n}$ for field $\K$
(that may involve parameters) and $\lambda$ is an additional variable,
then the Smith invariants $s_1, \ldots, s_n$ of
$\lambda \mat{I} - \mat{A}$ are the Frobenius invariants of $\mat{A}$
and $\mat{A}$ is similar to its Frobenius normal form,
$\oplus_{i=1}^n \mat{C}_{s_i}$, where $\mat{C}_s$ denotes the
companion matrix of polynomial $s$.  Thus we have comprehensive
Frobenius normal form as a corollary of CTSNF, however it is without
the similarity transform. It would be interesting to develop a
comprehensive Frobenius form in which each regime also includes a
similarity transform.

Parametric eigenvalue problems for $\mat{A}$ correspond to PLS for
$\lambda \mat{I} -\mat{A}$ with zero right hand side. Often eigenvalue
multiplicity is the concern.  The geometric multiplicity is available
from the Smith invariants, as for example on the diagonal of a
triangular Smith form.  Common roots of $2$ or more of the invariants
expose geometric multiplicity and square-free factorization of the
individual invariants exposes algebraic multiplicity.  Note that
square-free factorization may impose further restrictions on the
parameters.  Comprehensive treatment of square-free factorization is
considered in~\citep{LemMaz05}.

\subsection{Eigenvalue multiplicity example}
The following matrix, due originally to a question on
sci.math.num-analysis in 1990 by Kenton K.~Yee, is discussed
in~\citep{Cor02}.  We change the notation used there to avoid a clash
with other notation used here.  The matrix is
\[
  \mat{Y} = 
  \begin{bmatrix}
    z^{-1} & z^{-1} & z^{-1} & z^{-1} & z^{-1} & z^{-1} & z^{-1} & 0 \\
    1 & 1 & 1 & 1 & 1 & 1 & 0 & 1 \\
    1 & 1 & 1 & 1 & 1 & 0 & 1 & 1 \\
    1 & 1 & 1 & 1 & 0 & 1 & 1 & 1 \\
    1 & 1 & 1 & 0 & 1 & 1 & 1 & 1 \\
    1 & 1 & 0 & 1 & 1 & 1 & 1 & 1 \\
    1 & 0 & 1 & 1 & 1 & 1 & 1 & 1 \\
    0 & z & z & z & z & z & z & z \\
  \end{bmatrix}\>.
\]
One of the original questions was to compute its eigenvectors.  Since
it contains a symbolic parameter $z$, this is a parametric eigenvalue
problem which we can turn into a parametric linear system, namely to
present the nullspace regimes for $\lambda \mat{I} - \mat{Y}$.

Over $\F(z)[\lambda]$, after preconditioning, we get as the triangular
Smith form diagonal
$(1,1,1,1,1,\lambda^2 -1, \lambda^2-1, (\lambda^2 -1)f(\lambda))$,
where $f(\lambda) = \lambda^2 -(z + 6 + z^{-1})\lambda +7$.

\begin{remark}
  Without preconditioning, the Hermite form diagonal is instead
  $(1,1,1,1,\lambda -1, \lambda^2-1, (\lambda^2 -1), (\lambda +
  1)f(\lambda))$.
\end{remark}

The denominator of $\mat{U}$, $\mat{H}$ is a power of $z$, so the only
constraint is $z\neq 0$ which is already a constraint for the input
matrix.  We get regimes of rank~$5$ for $\lambda = \pm 1$, rank~$7$
for $\lambda$ being a root of $f$, and rank~$8$ for all other
$\lambda$.  In terms of the eigenvalue problem, we get eigenspaces of
dimension $3$ for each of $1$, $-1$ and of dimension $1$ for the two
roots of $f(\lambda)$.

To explore algebraic multiplicity, we can examine when $f$ has $1$ or
$-1$ as a root.  When $z$ is a root of ${z}^{2}+14\,z+1$, $f(\lambda)$
factors as $(\lambda+1)(\lambda+7)$ and when $z = 1$ we have
$f(\lambda) = (\lambda - 1)(\lambda - 7)$.  These factorizations may
be discovered by taking resultants of $f$ with $\lambda -1$ or
$\lambda +1$.






\subsection{Matrix Logarithm}
\label{sec:matrixlogarithm}

Theorem 1.28 of~\cite{Hig08} states conditions under which the matrix
equation $\exp(\mat{X}) = \mat{A}$ has so-called \emph{primary matrix
  logarithm} solutions, and under which conditions there are more.  If
the number of distinct eigenvalues $s$ of $\mat{A}$ is \emph{strictly
  less} than the number $p$ of distinct Jordan blocks of $\mat{A}$
(that is, the matrix $\mat{A}$ is \emph{derogatory}), then the
equation also has so-called \emph{nonprimary} solutions as well, where
the branches of logarithms of an eigenvalue $\lambda$ may be chosen
differently in each instance it occurs.

As a simple example of what this means, consider
\begin{equation}
  \mat{A} = \left[ \begin {array}{cc} a&1\\ \noalign{\medskip}0&a\end {array}
  \right] 
  \>.
\end{equation}
When we compute its matrix logarithm (for instance using the
\texttt{MatrixFunction} command in Maple), we find
\begin{equation}
  \mat{X}_A   = \left[
    \begin {array}{cc}
      \ln  \left( a \right) &{a}^{-1} \\
      \noalign{\medskip}0&\ln  \left( a \right)
    \end {array} \right] 
  \>.
\end{equation}
This is what we expect, and taking the matrix exponential (a
single-valued matrix function) gets us back to $\mat{A}$, as expected.
However, if instead we consider the derogatory matrix
\begin{equation}
    \mat{B} =  \left[ \begin {array}{cc} a&0\\ \noalign{\medskip}0&a\end {array}
 \right] 
\end{equation}
then its matrix logarithm as computed by \texttt{MatrixFunction} is
also derogatory, namely
\begin{equation}
  \mat{X}_B = \left[ \begin {array}{cc} \ln  \left( a \right) &0
    \\ \noalign{\medskip}0&\ln  \left( a \right) \end {array} \right] \>.
\end{equation}
Yet there are other solutions as well: if we add $2\pi i$ to the first
entry and $-2\pi i $ to the second logarithm, we unsurprisingly find
another matrix $\mat{X}_C$ which also satisfies
$\exp(\mat{X}) = \mat{B}$.  But adding $2\pi i$ to the first entry of
$\mat{X}_A$ while adding $-2\pi i$ to its second logarithm, we get
another matrix
\begin{equation}
  \mat{X}_D =  \left[ \begin {array}{cc} \ln  \left( a \right) +2\,i\pi&{a}^{-1}
    \\ \noalign{\medskip}0&\ln  \left( a \right) -2\,i\pi\end {array}
\right] 
\end{equation}
which has the (somewhat surprising) property that
$\exp(\mat{X}_D) = \mat{B}$, not $\mat{A}$.

This example demonstrates in a minimal way that the detailed Jordan
structure of $\mat{A}$ strongly affects the nature of the solutions to
the matrix equation $\exp(\mat{X}) = \mat{A}$.  This motivates the
ability of code to detect automatically the differing values of the
parameters in a matrix that make it derogatory.  To explicitly connect
this example to CTSNF, consider
\begin{equation}
  \mat{M} = \begin{bmatrix} a & b\\ 0 & a\end{bmatrix}
\end{equation}
so that $\mat{A}$ above is $\mat{M}_{b = 1}$ and
$\mat{B}= \mat{M}_{b = 0}$. The CTSNF applied to
$\lambda \mat{I} - \mat{M}$ produces two regimes, with forms
\begin{equation}
  \mat{H}_{b \neq 0} =
  \begin{bmatrix}
    1 & \lambda/b \\
    0 & (\lambda-a)^2
  \end{bmatrix},
  \mat{H}_{b=0} =
  \begin{bmatrix}
    \lambda - a & 0 \\
    0 & \lambda-a
  \end{bmatrix},
\end{equation} 
exposing when the logarithms will be linked or distinct.  Note that in
this case the Frobenius structure equals the Jordan structure.

\subsection{Model of infectious disease vaccine effect}


\cite{RahZou15} have made a model of vaccine effect when there are two
subpopulations with differing disease susceptibility and vaccination
rates.  Within this study stability of the model is a function of the
eigenvalues of a Jacobian $\mat{J}$.  Thus we are interested in cases
where the following matrix is singular.
\[
  \mat{A} = \lambda \mat{I} - \mat{J} = \begin{bmatrix}
  \lambda-w&  0&  -a&  -c\\
  0&  \lambda-x&  -b&  -d\\
  0&  0&  \lambda-a-y&   c\\
  0&  0&   b&  \lambda-d-z\\
\end{bmatrix}.
\]
Here $w,x$ are vaccination rates for the two populations, $y,z$ are
death rates, $a,d$ are within population transmission rates, and $b,c$
are the between population transmission rates. We have simplified
somewhat: for instance $a,b,c,d$ are transmission rates multiplied by
other parameters concerning population counts.  Stability depends on
the positivity of the largest real part of an eigenvalue.  For the
sake of reducing expression sizes in this example we will arbitrarily
set $y = z = 1/10$.  For the same reason we will skip right
multiplication by an R to achieve triangular Smith form.  Hermite form
$\mat{H}$ of $\lambda \mat{I} - \mat{J}$ will suffice, revealing the
eigenvalues that are wanted.
\[
  \mat{H} = 
\begin{bmatrix}
\lambda+w & 0& 0& -(ad-a\lambda-bc-a/10)/c\\
0& \lambda+x& 0& \lambda+1/10\\ 
0& 0& 1& (d-\lambda-1/10)/c\\ 
0& 0& 0&\lambda^2 +(1/5-a-d)\lambda +ad -cb -(1/10)d -(1/10)a +1/100\\
\end{bmatrix}.
\]
The discriminant of the last entry gives the desired information for
the application subject to the denominator validity: $ c \neq 0$.
When $c = 0$ the matrix is already in Hermite form, so again the
desired information is provided.

This example illustrates that often more than three parameters can be
easily handled.  In experiments with this model not reported here, we
did encounter cases demanding solution beyond the methods of this
paper.  On a more positive note, we feel that comprehensive normal
form tools could help analyze models like this when larger in scope,
for instance modeling $3$ or more subpopulations.

\subsection{The Kac-Murdock-Szeg\"o example}

\cite{CorMor17} report times for computation of the comprehensive
Jordan form for matrices of the following form, taken
from~\citep{Tre01}, of dimensions $2$ to about $20$:
\begin{equation}
  \mathrm{KMS}_n =  \left[ \begin {array}{ccccc} 1&-\rho&&&\\ \noalign{\medskip}\>-\rho\>&
{\rho}^{2}+1&\>-\rho&&\\ \noalign{\medskip}&\ddots&\ddots&\ddots&
\\ \noalign{\medskip}&&-\rho\>&{\rho}^{2}+1&\>-\rho\\ \noalign{\medskip}
&&&-\rho\>&1\end {array} \right]\>. 
\end{equation}
This is, apart from the $(1,1)$ entry and the $(n,n)$ entry, a
Toeplitz matrix containing one parameter, $\rho$.  The reported times
to compute the Jordan form were plotted in~\citep{CorMor17} on a log
scale, and looked as though they were exponentially growing with the
dimension, and were reported in that paper as growing exponentially.

The results of this paper show instead that polynomial time is
possible for this family, because there are only two parameters
($\rho$ and the eigenvalue parameter, say $\lambda$).  The Hermite
forms for these matrices are all (as far as we have computed) trivial,
with diagonal all $1$ except the final entry which contains the
determinant.  Thus all the action for the Jordan form must happen with
the discriminant of the determinant.  Experimentally, the discriminant
with respect to $\lambda$ has degree $n^2+n-4$ for KMS matrices of
dimension $n \ge 2$ (this formula was deduced experimentally by giving
a sequence of these degrees to the Online Encyclopedia of Integer
Sequences~\citep{Slo20}) and each discriminant has a factor
$\rho^{n(n-1)}$, leaving a nontrivial factor of degree $2n-4$ growing
only linearly with dimension.  The case $\rho=0$ does indeed give a
derogatory KMS matrix (the identity matrix).  The other factor has at
most a linearly-growing number of roots for each of which we expect
the Jordan form of the corresponding KMS matrix to have one block of
size two and the rest of size one.  We therefore see only polynomial
cost necessary to compute comprehensive Jordan forms for these
matrices, in accord with our theorem.
%

\section{Conclusions}
\label{sec:conc}

%
We have shown that using the CTSNF to solve parametric linear systems
is of cost polynomial in the dimension of the linear system and
polynomial in parameter degree, for problems containing up to three
parameters.  This shows that polynomially many regimes suffice for
problems of this type.  To the best of our knowledge, this is the
first method to prove a polynomial regime bound for parametric linear
systems with at most three parameters, certainly in the
constructible-regime model considered here.

It remains an open question whether, for linear systems with a fixed
number of parameters greater than three, a number of regimes suffices
that is polynomial in the input matrix dimension and polynomial degree
of the parameters, being exponential only in the number of parameters.

Through experiments with random matrices we have indications that the
worst case bounds we give are sharp, though we haven't proven this
point.  As the examples indicated, many problems will have fewer
regimes, and sometimes substantially fewer regimes.  We have not
investigated the effects of further restrictions of the type of
problem, such as to sparse matrices.

\section*{Acknowledgements}
This work was supported by the Natural Sciences and Engineering
Research Council of Canada and by the Ontario Research Centre for
Computer Algebra. The third author, L.~Rafiee Sevyeri, would like to
thank the Symbolic Computation Group (SCG) at the David R. Cheriton
School of Computer Science of the University of Waterloo for their
support while she was a visiting researcher there.

\newcommand{\doi}[1]{%
  \href{https://doi.org/#1}{DOI:\nolinkurl{#1}}%
}

\end{document}